\documentclass{amsart}
\usepackage{amsthm}
\usepackage{amsmath}
\usepackage{amssymb}

\theoremstyle{plain}
\newtheorem{theorem}{Theorem}[section]
\newtheorem{corollary}[theorem]{Corollary}
\newtheorem{lemma}[theorem]{Lemma}
\newtheorem{proposition}[theorem]{Proposition}

\theoremstyle{definition}
\newtheorem*{definition}{Definition}

\date{16 August 2010}

\title{Lebesgue-Fourier algebra of a hypergroup }
\author{M. Alaghmandan, R. Nasr-isfahani, and M. Nemati}

\begin{document}

\maketitle

\begin{abstract}
Let ${\mathcal L} A(H)$ be the Lebesgue-Fourier space of a hypergroup $H$ considered as a Banach space on $H$. In addition ${\mathcal L}A(H)$ is a Banach algebra with the multiplication inherited from $L^1(H)$. Moreover If $H$ is a regular Fourier hypergroup, ${\mathcal L}A(H)$  is a Banach algebra with pointwise multiplication. We study the amenability and character amenability of these two Banach algebras.\footnote{2000 {\it Mathematics Subject
Classification}: 43A62, 43A22.

 {\it Key words}: Hypergroups, Fourier spaces,  Lebesgue-Fourier algebras, amenability, character amenability.}
\end{abstract}

The theory of locally compact hypergroups in harmonic analysis was initiated independently by Dunkl \cite{dun} and
Jewett \cite{jew} in the early 1970s with small differences. In 1968 Pym \cite{pym} also considered
convolution structures which are close to this theory. A nice exposition of the subject can be found in \cite{ross}. We use the term {\em hypergroup} to refer to the locally compact hypergroups defined by Jewett \cite{jew}.

Fourier algebras over hypergroups have not been subject of much attention, since they do not need to form an algebra with pointwise multiplication. In his recent work, Muruganandam \cite{mu1} studies several hypergroups whose Fourier space forms a Banach algebra with pointwise multiplication. He defines a class of hypergroups called {\em regular Fourier hypergroups} where the associated Fourier space is a Banach algebra with its norm.

Lebesgue-Fourier algebras of locally compact groups were studied extensively by Ghahramani
and Lau in \cite{gl1}. Lebesgue-Fourier algebras are not only Segal algebras with convolution but also  abstract Segal algebras with respect to the Fourier algebra of a locally compact group.

After an overview of the preliminaries in section 1, in section 2, we define Lebesgue-Fourier space over a hypergroup.
We extend the definition and the main ideas of \cite{gl1} to hypergroup $H$; accordingly,
we consider ${\mathcal L}A(H)$ with the induced multiplication as a dense ideal of the algebra of
integrable functions, $L^1(H)$. Moreover, we show that the amenability of ${\mathcal L}A(H)$ with the multiplication induced
from $L^1(H)$ leads to the amenability and discrete being of $H$; also, $\alpha$-amenability of $H$ is equivalent to the $\phi_\alpha$-amenability
of ${\mathcal L}A(H)$ for each $\alpha \in \widehat{H}$. Section 3 is devoted to regular Fourier hypergroups $H$ when
Lebesgue-Fourier space with pointwise multiplication is a
dense ideal of the Fourier algebra $A(H)$ and therefore is a Banach algebra with pointwise multiplication, and as a result,
amenability of ${\mathcal L}A(H)$ with the pointwise product leads to compactness of $H$ and amenability of $A(H)$. Eventually, we show that
$A(H)$ and ${\mathcal L}A(H)$ are $\phi_x$-amenable for each $x\in H$.

\section{Preliminaries} \label{sec:prelim}

Let $H$ be a (locally compact Hausdorff) hypergroup which admits a
left Haar measure $m_H$. Here we use $dt$ instead of $dm_H(t)$. We
follow the definition of a hypergroup in \cite{jew}. For every
$x\in H$, suppose $\delta_x$ denotes the point measure at $x$. We
show the probability measure $\delta_x * \delta_y$ simply by $x y$
and $\check {x}$ denotes the involution of $x\in H$. Also $C(H)$
is the space of all continuous functions on $H$ and for all $x \in
H$ and $f\in C(H)$, $l_x f$ represents the (generalized) left
translation of $f$ by $x$,
$$
l_x f(y) := f(x y)=\int_{H} f(t) d(\delta_{x} * \delta_y)(t), \quad y\in H.
$$
As usual we consider $L^p(H)$ for $1\le p < \infty$ with respect to the left Haar measure of $H$ equipped with the translation operator whose norm is bounded from the above by 1. The translator can be extended to  $L^p (H)$ as an $L^1 (H)$-module by setting
$$
l_f g(x)= \int_{H}l_{\check{y}}g(x)f(y) dy
$$
for all $f \in L^1 (H)$ and $g \in L^p (H)$. If $f$ is a function on $H$, define $\tilde{f}$ by $\tilde{f}(x)=\overline{f(\check{x})}$ and $\check{f}$ by $\check{f}(x)=f(\check{x})$ for each $f \in L^p(H)$ and $x\in H$. The Banach space $L^{1}(H)$ is a Banach algebra with
$ f*g:=l_f g$
when $f,g\in L^1(H)$.

Let $\Sigma$ denote the set of equivalence classes of
representations of $H$ and $\lambda$ denote the left regular
representation of $H$ on $L^2(H)$ given by
\[ \lambda(x)f(y)=f(\check{x}y) \ {\rm for\ all}\ x,y\in H\ {\rm and\ for\ all \ } f\in L^2(H).\]
 Let $C^*(H)$ and
$C_{\lambda}^{*}(H)$ represent the full and reduced $C^{*}$-algebras
of $H$ respectively. The Von Neumann algebra associated to
$\lambda$ of H, namely, the bicommutant of $\lambda(L1(H))$ in
${\mathcal B}(L^2(H))$ is called the Von Neumannn algebra of $H$ and
is denoted by $VN(H)$ (see \cite{mu1,mu2}).  For any $f \in L^1(H)$ the norm of $C^*(H)$
is given by
$$
\|f\|_{C^*(H)} = \sup_{\pi \in \Sigma}\|\pi(f)\|_{{\mathcal B}({\mathcal H_{\pi}}))}
$$
when ${\mathcal H}_{\pi}$ is the Hilbert space related to  each $\pi \in \Sigma$;
accordingly we have
$$
\|f\|_{C_{\lambda}^*(H)} = \|\lambda(f)\|_{{\mathcal B}(L^2(H))}.
$$
\indent The Banach space dual of $C^*(H)$ is called the
Fourier-Stieltjes space and is denoted by $B(H)$. Let
$B_{\lambda}(H)$ denote the Banach space dual of the reduced
$C^*$-algebra $C_{\lambda}^*(H)$. In fact, $B_\lambda (H)$ can be
realized as a closed subspace of $B(H)$. The closed subspace
spanned by $\{f*\tilde{f} : f \in C_c(H)\}$ in $B_\lambda (H)$ is
called the Fourier space of H and is denoted by $A(H)$ \cite{mu1,mu2}. The Banach space dual of $A(H)$ can be identified with the Von Neumann algebra
of $H$ in the following way. For every $T\in VN(H)$ there exists a unique continuous linear functional $\varphi_T$ on $A(H)$ satisfying
\[ \varphi_T ((f*\tilde{g})\check{}\ )=\langle T(f),g\rangle_{L^2(H)}\  {\rm for\ all}\ f,g\in L^2(H).
\]
The mapping $T \rightarrow \varphi_T$ is a Banach space isomorphism between $VN(H)$ and $A(H)^*$. By an abuse of notation, we use $T$ to represent $\varphi_T$.

\section{Lebesgue-Fourier algebra for general hypergroups } \label{sec:lebesgue-fourier algebra}

Let $H$ be a hypergroup and let us define
$$
{\mathcal L}A(H):= L^1 (H) \cap A(H)
$$
and
$$
|||f|||=\|f\|_1 + \|f\|_{A(H)}.
$$
for each $f\in {\mathcal L}A(H)$. Similar to
\cite[lemma 2.1]{gl1}, we have

\begin{lemma}\label{114}
Given a hypergroup $H$, ${\mathcal L}A(H)$ is a dense left ideal of $L^1 (H)$.
\end{lemma}

\begin{proof}

 By \cite[corollary 2.12]{mu1}, we know that
\[ \left\{ \sum_{i=1}^{n} \phi_{i}*\tilde{{\varphi_{i}}} | \ \phi_i,\varphi_i \in L^2 (H)\ n \in \Bbb{N}\right\}\]
is dense in $A(H)$. But  $g*\phi$ is  an element of $L^2 (H)$ for each $g \in L^1 (H)$ and $\phi \in L^2 (H)$, so
$\sum_{i=1}^{n} g * (\phi_{i} * \tilde{\varphi_{i}} ) = \sum_{i=1}^{n}(g* \phi_{i}) * \tilde{\varphi_{i}}$  belongs to ${\mathcal L}A(H)$ for each $g \in L^1 (H)$
and $\sum_{i=1}^{n} \phi_{i} * \tilde{\varphi_{i}}\in {\mathcal L}A(H)$.

To show ${\mathcal L}A(H)$ is a dense set in
$L^1 (H)$ we refer to \cite[proposition 2.22]{mu1} and
\cite[lemma 2.1]{sk} to
generate a left bounded approximate identity for $L^1 (H)$ whose elements belong to ${\mathcal L}A(H)$.

\end{proof}

\begin{proposition}\label{116}
Let $H$ be a  hypergroup. Then $({\mathcal L}A(H),|||\cdot|||)$ with the induced multiplication from $L^1(H)$ is a Banach algebra.
\end{proposition}
\begin{proof}
Let $\{f_n\}$ be a Cauchy sequence  in ${\mathcal
L}A(H)$ which $\|\cdot\|_1$-converges to some $f\in
L^1(H)$ and $\|\cdot\|_{A(H)}$-converges  to some $f'\in A(H)$. There exists a
subsequence $\{f_{n_k}\}$ of $\{f_n\}$ which pointwise converges to $f$ almost everywhere. Moreover, based on   \cite[remark 2.9]{mu1}, $\|f\|_{\infty}\leq
\|f\|_{A(H)}$  for each $f \in {\mathcal L}A(H)$; therefore $f=f'$ almost everywhere.
So $|||f_n -f||| \rightarrow 0$.

Sine  $\|g * \phi\|_{2} \leq \|g\|_{1} \|\phi\|_{2}$ for each $g \in L^1 (H)$ and $\phi \in L^2 (H)$,
$$
\|g*(\phi*\tilde{\varphi})\|_{A(H)} \leq \|g\|_{1}
\|\phi*\tilde{\varphi}\|_{A(H)}\ {\rm for}\ g \in L^1(H)\ {\rm and}\ \phi,\varphi \in
L^2(H).
$$
Consequently, we have $||| g * f ||| \leq \|g\|_{1} |||f|||$ for each $f \in {\mathcal L}A(H)$. Hence,
\[ ||| g * f ||| \leq |||g|||\; |||f||| \]
when $f,g \in
{\mathcal L}A(H)$.

\end{proof}

\indent We know that ${\mathcal L}A(H)$ has an approximate identity
whose elements are bounded in $\|\cdot \|_1$ \cite[lemma 2.1]{sk}. Accordingly,
 in the following proposition we study the existence of an
 approximate identity for ${\mathcal L}A(H)$ bounded in $|||\cdot|||$.

\begin{proposition}\label{120}
Let $H$  be a hypergroup. The followings are equivalent:\\
\emph{(a)} $H$ is discrete;\\
\emph{(b)} ${\mathcal L}A(H) = L^1(H)$;\\
\emph{(c)} ${\mathcal L}A(H)$  has a bounded approximate identity.
\end{proposition}
\begin{proof} It is easy see $(a) \Rightarrow (b)$
and $(a) \Rightarrow (c)$.

$(b) \Rightarrow (a)$. We know $L^1(H) \subseteq A(H)
\subseteq C_0(H)$. We define $\iota: L^1(H)
\rightarrow C_0(H)$ where $\iota(f)$ is the function in $C_0(H)$ and
in the equivalence class of $f$. We will show that $\iota$
is continuous. Let $\{f_n\}$ be a sequence in $L^1(H)$ converging
to $f\in L^1(H)$; as a result there exists a subsequence
$\{f_{n_k} \}$ that converges to $f$ pointwise almost everywhere. If
$\{\iota(f_n)\}$ converges to $f'\in C_0(H)$, it follows $f'=\iota(f)$ almost everywhere. So by
closed graph theorem, $\iota$ is a continuous map.
Let $H$ be a non-discrete hypergroup, then $m_H(\{e\})=0$ by
\cite[theorem 7.1B]{jew}. By an argument similar to the proof
 of \cite[proposition 2.3]{gl1}, there is a bounded net
$\{f_\gamma\}$ in $L^1(H)$  when $\iota(f_\gamma) \rightarrow
\infty$, which contradicts the continuity of $\iota$ as a linear
map.

$(c) \Rightarrow (b)$. By assumption, there is a bounded right
approximate identity $(u_\gamma)$ such that $(u_\gamma)\subseteq{\mathcal L}A(H)$
and $|||u_\gamma|||\leq K$ for all $\gamma$. Since
$|||f*u_\gamma|||\leq\|f\|_1|||u_\gamma|||$ for all $f\in{\mathcal
L}A(H)$ and $\gamma$,  we have $|||f|||\leq K\|f\|_1$. On
the other hand, $\|f\|_1\leq |||f|||$. Thus
 the two norms $\|\cdot\|_1$ and $|||\cdot|||$ are equivalent on ${\mathcal L}A(H)$,  so
${\mathcal L}A(H)=L^1(H)$ by lemma \ref{114}.
\end{proof}

\begin{corollary}\label{124}
Let $H$ be a  hypergroup. If ${\mathcal L}A(H)$ with the multiplication induced from $L^1(H)$ is amenable then $H$ is discrete and amenable.
\end{corollary}
\begin{proof} If ${\mathcal L}A(H)$ is an amenable algebra, then based on  \cite[proposition 2.2.1]{ru}, it has a bounded approximate identity. So
by proposition  \ref{120} we have that $H$ is discrete. Also since ${\mathcal L}A(H)=l^1 (H)$, by \cite[proposition 4.9]{sk}
$H$ is amenable.
\end{proof}

\indent The converse of the preceding corollary is not true in general, because the amenability of $H$ does not show the amenability of
$L^1 (H)$ (see \cite{sk} and to consider a counter example see \cite{ar}).\\

\indent Let $H$ be a commutative hypergroup. The dual space of
 the hypergroup algebra, $L^1(H)$, can be identified with the usual
Banach space $L^\infty(H)$, and its structure space is homomorphic
to the character space of $H$, i.e.
$$
{\mathcal X}^b(H):=\Big{\{}\alpha\in C^b(H): \alpha(e)=1,
\alpha(xy)=\alpha(x)\alpha(y),~\hbox{for all}~ x,y\in H\Big\}
$$
equipped with the compact-open topology. ${\mathcal X}^b(H)$ is a
locally compact Hausdorff space. Let $\widehat{H}$ denote the set
of all hermitian characters $\alpha$ in ${\mathcal X}^b(H)$, i.e.
$\alpha(\check{x})=\overline{\alpha(x)}$ for every $x\in H$
with a Plancherel measure $\pi_H$. Note that $\widehat{H}$ in
general may not have the dual hypergroup structure  and a proper
inclusion in $supp(\pi_H)\subseteq \widehat{H}\subseteq{\mathcal
X}^b(H)$ is possible.

The Fourier-Stieltjes transform of $\mu\in M(H)$,
$\widehat{\mu}\in C^b(\widehat{H})$, is given by
$\widehat{\mu}(\alpha):=\int_H \overline{\alpha(x)}d\mu(x)$. Its
restriction to $L^1(H)$ is called the Fourier transform. We have
$\widehat{f}\in C_0(\widehat{H})$ for $f\in L^1(H)$, and the map that takes
$\alpha$ to $I(\alpha) = \ker(\phi_\alpha)$ is a bijection of
$\widehat{H}$ onto the space of all maximal ideals of $L^1(H)$,
where $\ker(\phi_\alpha)$ denotes the kernel of the homomorphisms
$\phi_\alpha(f)=\widehat{f}(\alpha)$ on $L^1(H)$ (see \cite{b.d}).

Let ${\mathcal A}$ be a Banach algebra and $\sigma({\mathcal A})$
be the set of all non-zero characters on ${\mathcal A}$.  Kaniuth,
Lau and Pym  \cite{klp,La1} introduced and studied  the concept of
{\it $\phi$-amenability} for Banach algebras as a generalization
of left amenability of  Lau algebras when $\phi\in\sigma({\mathcal
A})$.  ${\mathcal A}$ is {\it $\phi$-amenable} if there exists a
bounded net $(a_\gamma)$ in ${\mathcal A}$ such that
$\phi(a_{\gamma})\rightarrow 1$ and $\| aa_{\gamma}-
\phi(a)a_{\gamma}\| \rightarrow 0$ for all $a \in {\mathcal A}$.
Any such net is called a bounded approximate $\phi$-mean.

The notion of $\alpha$-amenable hypergroups was introduced and
studied in \cite{fls}. As shown in \cite{az2}, $H$ is
$\alpha$-amenable if and only if $L^1(H)$ is
$\phi_\alpha$-amenable. In the following theorem we explore the
connection between  $\alpha$-amenability of $H$ and
$\phi_\alpha$-amenability of ${\mathcal L}A(H)$.

\begin{theorem}\label{LA3}
Let $H$ be a commutative  hypergroup and let
$\alpha\in\widehat{H}$ be real-valued. Then the Lebesgue-Fourier
algebra, ${\mathcal L}A(H)$, is $\phi_\alpha$-amenable if and only if
$H$ is $\alpha$-amenable.
\end{theorem}
\begin{proof} Suppose $H$ is $\alpha$-amenable. Then
$L^1(H)$ is $\phi_\alpha$-amenable by \cite[ Theorem 1.1]{az2}.
Thus there is a bounded approximate $\phi_\alpha$-mean in
$L^1(H)$, say $(f_\gamma)$. Fix $h_0\in {\mathcal L}A(H)$ such that
$\phi_\alpha(h_0)=1$ and set $h_\gamma=f_\gamma *h_0\in{\mathcal
L}A(H)$ for all $\gamma$, and consequently,  for each $h\in {\mathcal
L}A(H)$ we have
\begin{eqnarray*}
|||h*h_\gamma-\phi_\alpha(h)h_\gamma|||&\leq&
\|h*f_\gamma-\phi(h)f_\gamma\|_1\;|||h_0|||\rightarrow 0
\end{eqnarray*}
and $\phi_\alpha(h_\gamma)=\phi_\alpha(f_\gamma)\rightarrow 1$.
Since $(f_\alpha)$ is $\|\cdot\|_1$-bounded, it follows that
$(h_\gamma)$ is $|||\cdot |||$-bounded. Thus ${\mathcal L}A(H)$ is
$\phi_\alpha$-amenable.

Conversely,  suppose that ${\mathcal L}A(H)$ is
$\phi_\alpha$-amenable. Then there is  a bounded approximate
$\phi_\alpha$-mean $(h_\gamma)$ in ${\mathcal L}A(H)$. Fix
$h_0\in{\mathcal L}A(H)$ such that $\phi_\alpha(h_0)=1$ and set
$f_\gamma=h_0*h_\gamma$ for all $\gamma$. Since ${\mathcal L}A(H)$ is
a left ideal in $L^1(H)$, we have
\begin{align*}
\|f*f_\gamma-\phi_\alpha(f)f_\gamma\|_1 &=  \|f*h_0
*h_\gamma-\phi(f)h_0*h_\gamma\|_1\\
&\leq \|f*h_0*h_\gamma-\phi_\alpha(f)\phi_\alpha(h_0)h_\gamma\|_1 \\
& + \|\phi_\alpha(f)\phi_\alpha(h_0)h_\gamma-\phi(f)h_0*h_\gamma\|_1\\
&\leq |||f*h_0*h_\gamma-\phi_\alpha(F*h_0)h_\gamma|||\\
& +|\phi(f)|\;|||\phi_\alpha(h_0)h_\gamma-h_0*h_\gamma|||\rightarrow
0,
\end{align*}
and $\phi_\alpha(f_\gamma)=\phi_\alpha(h_\gamma)\rightarrow 1$ for
each $f\in L^1(H)$ . Since $\|\cdot\|_1\leq |||\cdot|||$, it
follows that $(f_\gamma)$ is a $\|\cdot\|_1$-bounded approximate
$\phi_\alpha$-mean in $L^1(H)$, and $L^1(H)$ is $\phi_\alpha$-left
amenable. Thus $H$ is $\alpha$-amenable  by \cite[Theorem
1.1]{az2}.
\end{proof}

\section{Lebesgue-Fourier algebra for regular Fourier hypergroups }

Let $H$ be a commutative hypergroup, we define
$$
S=\{\alpha \in \widehat{H}\ |\ |\widehat{\mu}(\alpha)| \leq
\|\lambda(\mu)\| {\textrm{ for all}}\ \mu \in M(H)\}
$$
A non empty closed subset of $\widehat{H}$, see \cite{mu1}. Muruganandam has defined $(F)$ condition as following.
\begin{definition}\label{201}
Let $H$ be a commutative hypergroup. We say $H$ satisfies $(F)$ condition if there exists $M>0$ satisfying the following
$$
\textrm{ For every pair } \alpha,\alpha'\in S,\ \alpha\alpha'
\textrm{ belongs to } B_{\lambda}(H) \textrm{ and }
\|\alpha\alpha'\|_{B(H)} \leq M.
$$
\end{definition}

\indent  Some interesting results for commutative hypergroups which satisfy $(F)$ condition have been obtained in \cite{mu1}. We quote the
following corollary.

\begin{corollary}\label{205}
Let $H$ be a commutative hypergroup satisfying condition $(F)$. Then the Fourier space $A(H)$ is an algebra under pointwise product.
Moreover,
$$
\| f\cdot g \|_{A(H)} \leq M \|f\|_{A(H)}\; \|g\|_{A(H)} \ {\rm for}\ f,g \in A(H).
$$
In particular if $M=1$, then $A(H)$ forms a Banach algebra. Similar results hold for $B_{\lambda}(H)$.
\end{corollary}

\indent This corollary led Muruganandam to define (regular)
Fourier hypergroups in \cite{mu1}.

\begin{definition}\label{210}
A hypergroup $H$ is called a {\em Fourier hypergroup} if
\begin{enumerate}
    \item The Fourier space $A(H)$ forms an algebra with pointwise product.
    \item There exists a norm on $A(H)$ which is equivalent to the
original norm with respect to which $A(H)$ forms a Banach algebra.
\end{enumerate}
A hypergroup is called a {\em regular Fourier hypergroup} if $A(H)$ is a
Banach algebra with its original norm and pointwise product.
\end{definition}

\indent As we have seen in corollary \ref{205}, all commutative
hypergroups which satisfy $(F)$ for some $M>0$ are Fourier
hypergroups. If $M=1$ then $H$ is a regular Fourier hypergroup.
Consequently, several (regular) Fourier hypergroups have been
introduced in \cite{mu1} section 4.1. We will scope on regular
hypergroups to pursue results for ${\mathcal L}A(H)$.

\begin{proposition}\label{220}
Let $H$ be a regular Fourier hypergroup. Then ${\mathcal L}A(H)$ is a dense ideal in $A(H)$.
\end{proposition}

\begin{proof}
 For each $f \in {\mathcal L}A(H)$ and $\phi \in
A(H)$, $\phi \cdot f$ is in $L^1(H)$, since
$f$ belongs to $C_0(H)$. Moreover, regular Fourier
hypergroup property of $H$ implies $\phi \cdot
f \in A(H)$. Because $A(H) \cap C_c(H)$ is dense in
$A(H)$, ${\mathcal L}A(H)$ is dense in $A(H)$, see
\cite[corollary 2.12]{mu1}.
\end{proof}

\begin{proposition}
Let $H$ be a regular Fourier hypergroup. Then $({\mathcal L}A(H),|||\cdot|||)$ with the pointwise multiplication is a Banach algebra.
\end{proposition}

\noindent {\it Proof.} As we have seen in the proof of proposition
\ref{116}, ${\mathcal L}A(H)$ is a Banach space. Since
$\|\cdot\|_{\infty}\leq\|\cdot\|_{A(H)}$, we have
$$
\|\phi \cdot f\|_{1} \leq \|\phi\|_{A(H)} \|f\|_1
$$
for each $\phi \in {\mathcal L}A(H)$ and $f \in L^1(H)$. Hence
$$
|||g\cdot f||| \leq |||g||| \; |||f|||
$$
for each $f,g \in {\mathcal L}A(H)$. $\hfill\Box$\\

\begin{proposition}\label{222}
Let $H$  be a regular Fourier hypergroup. The followings are equivalent:\\
\emph{(a)} $H$ is compact;\\
\emph{(b)} ${\mathcal L}A(H) = A(H)$;\\
\emph{(c)} ${\mathcal L}A(H)$ with the pointwise product has a bounded approximate identity.
\end{proposition}

\noindent {\it Proof.} Clearly $(a) \Rightarrow (b)$
and $(a) \Rightarrow (c)$.

$(b) \Rightarrow (a)$. We know $A(H) \subseteq L^1(H)$.
We define $i: A(H) \rightarrow L^1(H)$ where
$i(f)$ is the equivalent class of $f$ in $L^1(H)$. We
show that $i$ is continuous. Let $\{f_n\}$ be a convergent
sequence in $A(H)$ to some $f\in A(H)$ which implies $f_n
\rightarrow f$ uniformly. If $\{i(f_n)\}$ converges
to $f'\in L^1(H)$, and as a result there is a subsequence
$i(f_{n_k})$ which converges to $f'$ pointwise almost everywhere. Therefore $f'=i(f)$ almost everywhere. So by closed graph
theorem, $i$ is a continuous map. Let $H$ be a
non-compact hypergroup, then $m_H(H)=\infty$ by
\cite[theorem 7.2B]{jew}. By the same argument as in the proof of \cite[proposition 2.6]{gl1} we can make a bounded net
$\{f_\gamma\}$ in $A(H)$  when $\|i(f_\gamma)\|_1 \rightarrow \infty$,
which contradicts the continuity of $i$ as a linear map.

$(c) \Rightarrow (b)$. By assumption, there is a bounded
approximate identity, say $(u_\gamma)\subseteq{\mathcal L}A(H)$,
with $|||u_\gamma|||\leq K$ for all $\gamma$. Since $|||f\cdot
u_\gamma|||\leq\|f\|_{A(H)}|||u_\gamma|||$ for all $f\in{\mathcal
L}A(H)$ and $\gamma$, it follows that $|||f|||\leq K\|f\|_{A(H)}$.
On the other hand $\|f\|_{A(H)}\leq |||f|||$. Thus
 the two norms $\|\cdot\|_{A(H)}$ and $|||\cdot|||$ are equivalent on ${\mathcal L}A(H)$, and so
${\mathcal L}A(H)=A(H)$; this is because ${\mathcal L}A(H)$ is dense
in $A(H)$ under $\|\cdot\|_{A(H)}$.  $\hfill\Box$\\

\begin{corollary}\label{225}
Let $H$ be a regular Fourier hypergroup and ${\mathcal L}A(H)$ with the pointwise product be amenable. Then $H$ is compact,
and the Fourier algebra $A(H)$ is amenable.
\end{corollary}
\noindent {\it Proof.} Suppose that ${\mathcal L}A(H)$ is an amenable
algebra. Then by proposition 2.2.1 of \cite{ru}, it has a bounded
approximate identity. So, by proposition \ref{222},
$H$ is compact. Also since ${\mathcal L}A(H)=A(H)$, the norms
$|||\cdot|||$
and $\| \cdot \|_{A(H)}$ are equivalent by the open mapping theorem. So $A(H)$ is amenable. $\hfill\Box$\\

\indent Let $H$ be a regular Fourier hypergroup and let $x\in H$.
Then the functional $\phi_x$ given by $\phi_x(f)=f(x)$ for all
$f\in A(H)$  belongs to $\sigma(A(H))$ by \cite[proposition 2.22]{mu1}.

\begin{theorem}\label{LA3}
Let $H$ be a regular Fourier hypergroup. Then $A(H)$ is
$\phi_x$-amenable for all $x\in H$.
\end{theorem}
\noindent {\it Proof.} Fix $x\in H$ and let ${\mathcal U}$ denote a net of relatively compact
neighborhoods of $e$ in $H$. For $U\in{\mathcal U}$, define
$f_U\in A(H)$ by
\begin{eqnarray*}
f_U(y) &=&m_H(U)^{-1}\langle \lambda(y)1_{\check{x}U},1_{U}\rangle_{L^2(H)}\\
&=&m_H(U)^{-1}\int_{U}1_{\check{x}U}(\check{y}z)dz\\
&=&m_H(U)^{-1}m_H(y\check{x}U\cap U)
\end{eqnarray*}
for all $y\in H$, where $\lambda$ is the left regular
representation of $H$ on $L^2(H)$. Recall that $VN(H)$ is
canonically identified with the dual space $A(H)^*$ by the paring
$\langle \lambda(y), f\rangle=f(y)$ (cf. \cite[proposition 2.21]{mu1}). Then
\[ \|f_U\|_{A(H)}\leq m_H(U)^{-1}\|1_{\check{x}U}\|_2\|1_U\|_2=1. \]
Let $F$ be  a
weak$^*$ cluster point of $(f_U)_U$ in $A(H)^{**}$. Then
$$
F(\phi_x)=\lim_U\langle f_U, \phi_x\rangle=f_U(x)=1,
$$
and moreover
\[ F({\lambda(y)})=  \lim_U \frac{m_H(y\check{x}U\cap
U)}{m_H(U)}=0 \]
for $y\neq x$. Therefore
\begin{eqnarray*}
F({\lambda(y)}\cdot f)&=&\lim_U\langle f_U,{\lambda(y)}\cdot f\rangle\\
&=&\lim_U\langle{\lambda(y)},ff_U\rangle\\
&=& f(y) \; \lim_U \frac{m_H(y\check{x}U\cap U)}{m_H(U)}\\
&=& f(y)\; F(\lambda(y))\\
&=& \phi_x(f)\; F(\lambda(y))
\end{eqnarray*}
for all $y\in H$ and $f\in A(H)$. Since  $\{ \lambda(y): \;y\in H\}$
generates $VN(H)$, we conclude that $F(T\cdot f)=\phi_x(f)\; F(T)$
for all $T\in VN(H)$ and $f\in A(H)$. This completes the
proof.$\hfill\Box$

\begin{corollary}\label{LA4}
Let $H$ be a regular Fourier hypergroup and ${\mathcal L}A(H)$ is
equipped with the pointwise product. Then ${\mathcal L}A(H)$ is
$\phi_x$-amenable for all $x\in H$.
\end{corollary}
\noindent {\it Proof.} First note that since ${\mathcal L}A(H)$ is
dense in $A(H)$, $\phi_x|_{{\mathcal L}A(H)}\neq
0$ and $\phi_x|_{{\mathcal L}A(H)}$ belongs to $\sigma({\mathcal L}A(H))$.
In addition, since $A(H)$ is $\phi_x$-amenable, there is a
bounded approximate $\phi_x$-mean $(f_\gamma)$ in $A(H)$. Fix
$h_0\in {\mathcal L}A(H)$ such that $\phi_x(h_0)=1$ and set
$h_\gamma=f_\gamma h_0\in{\mathcal L}A(H)$ for all $\gamma$. Thus
for each $h\in {\mathcal L}A(H)$ we have
$$
|||hh_\gamma-\phi_x(h)h_\gamma|||\leq\|hf_\gamma-\phi_x(h)f_\gamma\|_{A(H)}\;|||h_0|||\rightarrow
0
$$
and $\phi_x(h_\gamma)=\phi(f_\gamma)\rightarrow 1$. Thus ${\mathcal
L}A(H)$ is $\phi_x$-amenable.$\hfill\Box$\\

\noindent {\bf Acknowledgment.} We thank  Morteza Kamalian and Professor Hossein Namazi for their useful comments.

 \vspace{5mm}

 \noindent
   M. Alaghmandan\\
   Department of Mathematics and Statistics,\\
     University of Saskatchewan,\\
Saskatoon, SK, S7N 5E6, Canada\\
       {\bf Email:} mahmood.a@usask.ca\\

    \noindent
       R. Nasr-isfahani and M. Nemati\\
   Department of Mathematical Sciences,\\
     Isfahan Uinversity of Technology,\\
      Isfahan 84156-83111, Iran\\
       {\bf Emails:} isfahani@cc.iut.ac.ir,  m.nemati@math.iut.ac.ir

       \end{document}